\documentclass[11pt,reqno]{amsart}
\usepackage{mathrsfs}
\usepackage{amssymb}
\usepackage{amsfonts}
\usepackage{amsmath}
\usepackage{txfonts}
\usepackage{dsfont}
\usepackage{amscd}
\usepackage[colorlinks, linkcolor=blue, anchorcolor=blue, citecolor=blue]{hyperref}
\begin{document}
\setlength{\oddsidemargin}{0cm} \setlength{\evensidemargin}{0cm}
\baselineskip=20pt

\theoremstyle{plain} \makeatletter
\newtheorem{theorem}{Theorem}[section]
\newtheorem{Prop}[theorem]{Proposition}
\newtheorem{lemma}[theorem]{Lemma}
\newtheorem{Corollary}[theorem]{Corollary}
\newtheorem{remark}[theorem]{Remark}

\theoremstyle{definition}
\newtheorem{notation}[theorem]{Notation}
\newtheorem{exam}[theorem]{Example}
\newtheorem{prop}[theorem]{Proposition}
\newtheorem{conj}{Conjecture}
\newtheorem{prob}[theorem]{Problem}
\newtheorem{Rem}[theorem]{Remark}
\newtheorem{claim}{Claim}
\newtheorem{Definition}[theorem]{Definition}

\newcommand{\SO}{{\mathrm S}{\mathrm O}}
\newcommand{\SU}{{\mathrm S}{\mathrm U}}
\newcommand{\Sp}{{\mathrm S}{\mathrm p}}
\newcommand{\so}{{\mathfrak s}{\mathfrak o}}
\newcommand{\Ad}{{\mathrm A}{\mathrm d}}
\newcommand{\m}{{\mathfrak m}}
\newcommand{\g}{{\mathfrak g}}
\newcommand{\h}{{\mathfrak h}}

\numberwithin{equation}{section}
\title[Naturally reductive $(\alpha_1, \alpha_2)$ metrics]{Naturally reductive $(\alpha_1, \alpha_2)$ metrics}

\author{Ju Tan}
\address[Ju Tan]{School of Mathematics and Physics,
Anhui University of Technology, Maanshan, 243032, P.R. China}\email{tanju2007@163.com}
\author{Ming Xu$^*$}
\address[Ming Xu] {School of Mathematical Sciences,
Capital Normal University,
Beijing 100048,
P.R. China}
\email{mgmgmgxu@163.com}

\thanks{$^*$Ming Xu is the corresponding author.}

\begin{abstract} Let $F$ be a homogeneous $(\alpha_1,\alpha_2)$ metric on the reductive homogeneous manifold $G/H$. Firstly, we characterize the natural reductiveness of $F$ as a local $f$-product between naturally reductive Riemannian metrics.
Secondly, we prove the equivalence among several properties of $F$ for its mean Berwald curvature and S-curvature. Finally, we find an explicit flag curvature formula when $F$ is naturally reductive.
%
%

\end{abstract}

\maketitle
\section{Introduction}

In \cite{KN1963}, Kobayashi and Nomizu introduced an important subclass of homogenous Riemannian metrics, which are called naturally reductive metrics and generalize symmetric and normal homogeneous metrics.
Further generalization was given in \cite{KV1991}, where Kowalski and Vanhecke proposed the geodesic orbit property, i.e., all geodesics are the orbits of one-parameter isometric subgroups. Indeed, before \cite{KV1991}, Kaplan had already constructed the first example of
 geodesic orbit space which is not naturally reductive \cite{Ka1983}. The
natural reductiveness in homogeneous Riemannian geometry
has been extensively studied in homogeneous Riemannian geometry for many decades. For example,
in \cite{DZ1979}, D'Atri and Ziller investigated the left invariant naturally reductive Riemannian metrics
on compact Lie groups and compact homogeneous spaces, they gave a classification of naturally reductive Einstein
metrics on compact simple Lie groups and obtained lots of new examples of Einstein metrics, and
in \cite{Go1985}, Gordon studied the structure of naturally reductive Riemannian metrics
on non-compact Lie groups and classified naturally reductive nilmanifolds. The natural reductiveness of weakly symmetric spaces, D'Atri spaces, etc, are also explored. See (\cite{BV1996}\cite{Gr1972}\cite{Ko1983}\cite{KV1991}\cite{Wo1968}\cite{Zi1996})
and the references therein.

In this paper, we discuss the natural reductiveness in Finsler geometry, which was given from different view points, by Latifi in \cite{La2007}, and by S. Deng and Z. Hou
in \cite{DH2010}.
The equivalence between these two definitions was proved in \cite{ZYD2022}. It should be remarked that,
unlike the situation in Riemannian geometry, the Finslerian natural reductiveness can not be implied from the normal homogeneity, and in some sense, it is even stronger, because naturally reductive Finsler metrics are Berwaldian.

The first goal of this paper is to characterize the natural reductiveness of non-Riemannian
$(\alpha_1,\alpha_2)$ metric. The $(\alpha_1,\alpha_2)$ metrics are defined in \cite{DX2016}.
We believe that they generalize Randers metrics \cite{Ra1941} and $(\alpha,\beta)$-metrics\cite{Ma1992}, and then share some similar properties \cite{DX2016}\cite{XD2021}\cite{XM2021}. The $f$-product
in \cite{CS2004} is a special case of Berwald $(\alpha_1,\alpha_2)$ metric. It is defined as a Finsler metric $F=f(\alpha_1,\alpha_2)$ on $M_1\times M_2$, where each $\alpha_i$ is a Riemannian metric on $M_i$. A general $(\alpha_1,\alpha_2)$ metric can be similarly presented as $F=f(\alpha_1,\alpha_2)$, but each $\alpha_i$ is only defined
on a (possibly non-integrable) distribution $\mathcal{V}_i$ in $TM$ with $TM=\mathcal{V}_1+\mathcal{V}_2$.

From a geometric point of view,
we prove the following explicit and simple description.

\begin{theorem}\label{main-thm-1}
Any non-Riemannian naturally reductive homogeneous $(\alpha_1,\alpha_2)$ metric is locally isometric to an $f$-product between two naturally reductive
Riemannian metrics.
\end{theorem}

For connected and simply connected manifolds,
the local $f$-product description in Theorem \ref{main-thm-1} is global and the inverse of Theorem \ref{main-thm-1} is valid (see Corollary \ref{cor-1}). With some minor changes, the proof of Theorem \ref{main-thm-1} might argue a similar description for non-Riemannian naturally reductive $(\alpha,\beta)$ manifold.

Notice that naturally reductive Finsler metrics are Berwald \cite{DH2010}. These results are coherent to Szab\'{o}'s description for Berwald metrics, i.e., some local ``product" among
Riemannian manifolds and non-Riemannian affine symmetric spaces \cite{Sz1981}. Here the speciality of $(\alpha_1,\alpha_2)$ metric type specifies how the ``product" is built and constrains the product factor and the factor number.


Switching to the algebraic point of view, then we prove
that, if the non-Riemannian homogeneous $(\alpha_1,\alpha_2)$ metric $F$ is defined using the homogeneous Riemannian metric $\alpha$, then the natural
reductiveness of $F$ is in fact inherited from $\alpha$ (see Corollary \ref{cor-2}). This byproduct generalizes the corresponding result for $(\alpha,\beta)$ metrics (see Theorem 3.7 in \cite{Pa2022}). Meanwhile, we see that
a non-Riemannian homogeneous $(\alpha_1,\alpha_2)$ manifold $(M,F)$ may achieve its natural reductiveness for diffferent homogeneous space representations $M=G/H$.
The assertion that $M$ is locally an $f$-product does not imply $G$ is also locally a product (see Example \ref{example}).

The second task of this paper is to discuss the curvature properties of a homogeneous $(\alpha_1,\alpha_2)$ manifold.

The S-curvature was found
by Z. Shen when he proved the volume comparison theorem in Finsler geometry \cite{Sh1997}. The Hessian of the S-curvature $S(x,y)$ for its $y$-entries essentially provides the mean Berwald curvature (E-curvature in short). E-curvature helps us study the celebrated Landsberg Conjecture \cite{LZ2020}. In general Finsler geometry, there are many properties related to the S-curvature (for example, the vanishing, constant, isotropic, and weakly isotropic S-curvature conditions),  and there are many others related to the E-curvature (for example, the vanishing, constant and isotropic E-curvature conditions).
These properties are closely related, and some are equivalent
in homogeneous Finsler geometry  \cite{XD2015}. If we specify the homogeneous metric to be of Randers or $(\alpha,\beta)$ type, even more equivalences can be proved \cite{De2009}\cite{DW2010}\cite{WZ2021}. We transport these observation to homogeneous $(\alpha_1,\alpha_2)$ metrics and prove the following theorem.

\begin{theorem}\label{main-thm-2}
For a homogeneous $(\alpha_1,\alpha_2)$ metric, the following are equivalent:
\begin{enumerate}
\item it has weakly isotropic S-curvature;
\item it has vanishing S-curvature;
\item it has isotropic E-curvature;
\item it has vanishing  E-curvature.
\end{enumerate}
\end{theorem}

Flag curvature generalizes sectional curvature in Riemannian geometry \cite{BCS2000}. But its calculation in general and homogeneous Finsler geometry is much harder. However, for a naturally reductive Finsler metric, the vanishing of the spray vector field greatly simplifies the homogeneous flag curvature formula \cite{Hu2015}. For a naturally reductive $(\alpha_1,\alpha_2)$ metric of the form $F=\alpha\phi(\alpha_2/\alpha)$, we present an explicit flag curvature formula (see Theorem \ref{main-thm-3}).

The arrangement of this paper is as the following. In Sect. 2, we summarize some
preliminaries in general and homogeneous Finsler geometry.
In Sect. 3, we prove Theorem \ref{main-thm-1}. In Sect. 4, we prove Theorem \ref{main-thm-2} and calculate the flag curvature formula for a naturally reductive $(\alpha_1,\alpha_2)$ metric.

\section{Preliminaries}\label{preliminaries}

In this section, we summarize some fundamental knowledge on general Finsler geometry from \cite{BCS2000}\cite{Sh2001},
recall the basic settings of homogeneous Finsler geometry in
\cite{De2012}. The notion of $(\alpha_1,\alpha_2)$ metric is from \cite{DX2016}.

\subsection{Minkowski norms and Finsler metrics}
\begin{Definition}
Let $\mathbf{V}$ be a finite dimensional real vector space. A {\it Minkowski norm} on $\mathbf{V}$ is a continuous function $F:\mathbf{V}\rightarrow[0,+\infty)$   satisfying the following properties:
\begin{enumerate}
  \item regularity: $F$ is positive and smooth on $\mathbf{V}\backslash\{0\}$;
  \item positive 1-homogeneity: $F(\lambda y)=\lambda F(y)$ for any $y\in \mathbf{V}$  and $\lambda\geq0$;
  \item strong convexity: for any $y\in V\backslash\{0\}$, the {\it fundamental tensor}
\[g_y(u,v):=\frac{1}{2}\frac{\partial^2}{\partial s\partial t}[F^2(y+su+tv)]_{s=t=0}                                                  \]
is an inner product on $V$.
\end{enumerate}
\end{Definition}

Sometimes, the {fundamental tensor} is referred to the positive definite Hessian matrix $(g_{ij})=(\frac12\frac{\partial^2}{\partial y^i\partial y^j}F^2)$
with respect to the linear coordinate $y=y^ie_i\neq0$ for a basis $\{e_i,\forall i\}$ on
$\mathbf{V}$.  We use it and its inverse $(g^{ij})$ to move indices up and down.


\begin{Definition}
A {\it Finsler metric} on a smooth manifold $M$ is a continuous function $F: TM\rightarrow [0, +\infty)$ such that
\begin{enumerate}
  \item $F$ is smooth on the slit tangent bundle $TM \backslash\{0\}$;
  \item the restriction of $F$ to each tangent space $T_xM$ is a Minkowski norm.
\end{enumerate}
\end{Definition}
We also call the pair $(M,F)$ a {\it Finsler manifold} or a {\it Finsler space}.

For example,
a Finsler metric is {\it Riemannian} when the fundamental tensor
$(g_{ij}(x,y))$ depends on $x\in M$ only, for any standard local coordinate $x=(x^i)\in M$ and $y=y^i\partial_{x^i}\in T_xM$, or equivalently, the {\it Cartan tensor} $$C_y(u,v,w)=\frac12\frac{d}{dt}|_{t=0}g_{y+tw}(u,v)
=\frac14\frac{\partial^3}{\partial r\partial s\partial t}|_{r=s=t=0}F^2(y+ru+sv+tw)$$vanishes
everywhere.

A {\it Randers metric} has the form $F=\alpha+\beta$, where $\alpha$ is a Riemannian metric and
$\beta$ is an one-form with pointwise $\alpha$-norm smaller than 1. An {\it $(\alpha,\beta)$ metric} has the form $F=\alpha\phi(\beta/\alpha)$, where $\phi(s)$ is a positive smooth function of one variable.

An {\it $(\alpha_1,\alpha_2)$ metric} on the smooth manifold $M$
is defined as the following \cite{DX2016}. Let $\alpha$ be a Riemannian metric on $M$ and $TM=\mathcal{V}_1+\mathcal{V}_2$ be an
$\alpha$-orthogonal decomposition of $TM$ into two
linear sub-bundles. Then a Finsler metric is called an $(\alpha_1,\alpha_2)$
metric if $F(y)$ only depends on the values $\alpha(y_1)$ and $\alpha(y_2)$, where $y=y_1+y_2\in TM$ with $y_i\in\mathcal{V}_i$. Using the notation
$\alpha_i(y)=\alpha(y_i)$, we may present an $(\alpha_1,\alpha_2)$ metric as $F=\sqrt{L(\alpha^2_1,\alpha^2_2)}$ for some positively 1-homogeneous smooth real function $L:[0,+\infty)\times[0,+\infty)\rightarrow[0,+\infty)$. Sometimes, we present an $(\alpha_1,\alpha_2)$ as $F=\alpha\phi(\alpha_2/\alpha)$ or $F=\alpha\psi(\alpha_1/\alpha)$, mocking the $(\alpha,\beta)$ metrics. In each tangent space, the $(\alpha_1,\alpha_2)$ metric $F$ defines a Minkowski norm with linear $O(n_1)\times O(n_2)$ symmetry, which is called an {\it $(\alpha_1,\alpha_2)$ norm}. When $\alpha_i$ is the Riemannian metric on $M_i$, then $(\alpha_1,\alpha_2)$ metric $F=\sqrt{L(\alpha^2_1,\alpha^2_2)}$ on $M_1\times M_2$ is a Berwald $(\alpha_1,\alpha_2)$ metric (see Example 4.3.1 in \cite{CS2004}). We will call it an {\it $f$-product metric}.

A Minkowski norm or a Finsler metric is called reversible, if opposite (tangent) vectors have the same length. Obviously,
all $(\alpha_1,\alpha_2)$ metrics (including Riemannian metrics) are reversible, and most $(\alpha,\beta)$ metrics (including Randers metrics) are not. The $(\alpha,\beta)$ and $(\alpha_1,\alpha_2)$ metrics can achieve the maximal non-Euclidean linear symmetry degree in each tangent space, i.e., $O(n-1)$ and $O(n_1)\times O(n_2)$ respectively, where
$n=n_1+n_2$ is the dimension of the manifold. So the calculation and formulae for these metrics have relatively low complexity.

\subsection{E-curvature, S-curvature and flag curvature}

For a Finsler manifold $(M,F)$, the {\it geodesic spray} is the smooth tangent vector field $\mathbf{G}$ on $TM\backslash0$,
which can be locally presented as $$\mathbf{G}=y^i\partial_{x^i}-2\mathbf{G}^i\partial_{y^i}$$
with
$\mathbf{G}^i=\frac14 g^{il}([F^2]_{x^ky^l}y^k-[F^2]_{x^l})$ for each $i$, for any standard local coordinate.

We call $B^i_{jkl}=\frac{\partial^3\mathbf{G}^i}{\partial y^j\partial y^k\partial y^l}$ the {\it Berwald curvature} and
$E^{ij}=\frac12 B^k_{ijk}$ the {\it mean Berwald curvature} or {\it E-curvature}. We say $(M,F)$ has isotropic E-curvature if there exists a function $c(x)$ on $M$ such that
$E_{ij}=\frac{n+1}{2}c(x) F^{-1} h_{ij}$, in which $h_{ij}=FF_{y^iy^j}$ is the angular form.

Denote the B.H. measure of $(M,F)$ by $dV_{BH}=\sigma(x) dx^1\cdots dx^n$, in which
$$\sigma(x)=\frac{Vol(B_n(1))}{Vol\{(y^i)\in\mathbb{R}^n|
F(x,y^i\partial_{x^i})<1\}}.$$
Then the S-curvature $S:TM\backslash0\rightarrow\mathbb{R}$ is the derivative
of the distortion function
$$\tau(x,y)=\ln\frac{\sqrt{\det(g_{ij}(x,y))}}{\sigma(x)}$$
 in the direction of the geodesic spray. More explicitly,
the S-curvature can be presented by
\begin{equation}\label{010}
S(x,y)=\frac{\partial \mathbf{G}^i}{\partial y^i}(x,y)
-y^i\frac{\partial}{\partial x^i}\ln|\sigma(x)|,
\end{equation}
for any standard local coordinate.

We say $(M,F)$ has weakly isotropic S-curvature if there exists a function $c(x)$ and  a one-form $\epsilon$ on $M$, such that
\begin{equation}\label{101}
S(x,y)=(n+1)c(x)F(x,y)+\epsilon(x,y)
\end{equation}
is satisfied
at every $(x,y)\in TM\backslash{0}$. We say a weak isotropic S-curvature metric $F$ is isotropic, if the one-form $\epsilon$ in
(\ref{101}) vanishes. We say a isotropic S-curvature metric $F$
is of constant or vanishing S-curvature if we further have $c(x)\equiv\mathrm{const}$ or $c(x)\equiv0$
respectively.

\begin{remark} In this paper, the S-curvature is always referred to the B.H. measure. Generally speaking, it can be defined for any smooth measure on a Finsler manifold. Some S-curvature properties, like the vanishing, constant or isotropic S-curvature properties depend on the choice of the smooth measure, and some others, like the weakly isotropic S-curvature property does not.
\end{remark}

On a Finsler manifold $(M,F)$, the {\it flag curvature} is defined for each flag triple $(x,y,\mathbf{P})$ with $x\in M$, $y\in T_xM\backslash\{0\}$ and a tangent plane $\mathbf{P}\subset T_xM$ containing $y$. Suppose $\mathbf{P}=\mathrm{span}\{y,v\}$, then the flag curvature of this triple is
$$K(x,y,\mathbf{P})=\frac{g_y(R_y(v),v)}{
g_y(y,y)g_y(v,v)-g_y(y,v)^2}.$$
Here $R_y=R^i_k dx^k\partial_{x^i}:T_xM\rightarrow T_xM$
is the Riemann curvature determined by
$$R^i_k=2\frac{\partial\mathbf{G}^i}{\partial x^k}-y^j\frac{\partial^2\mathbf{G}^i}{\partial x^j\partial y^k}+2\mathbf{G}^j\frac{\partial^2\mathbf{G}^i}{
\partial y^j\partial y^k}-\frac{\partial \mathbf{G}^i}{\partial y^j}\frac{\partial \mathbf{G}^j}{\partial y^k}.$$

\subsection{Homogeneous Finsler space}

Let $(M,F)$ be a  Finsler manifold. We call $(M,F)$ or $F$
{\it homogeneous} if
the isometry group $I(M,F)$ acts transitively on $M$ \cite{De2012}. Since $I(M,F)$ is a Lie transformation group \cite{DH2002}\cite{MS1939}, we may present $M$ as the smooth coset space $G/H$ for any Lie subgroup $G$ in $I(M,F)$, where $H=\{g\in G| g\cdot o=o\}$ is the isotropy subgroup at $o=eH\in G/H$.


Denote $\mathfrak{g}=\mathrm{Lie}(G)$ and $\mathfrak{h}=\mathrm{Lie}(H)$. Since the $G$-action on $M$ is effective, $H$ is compactly imbedded. So we can find an $\mathrm{Ad}(H)$-complement $\mathfrak{m}$ of $\mathfrak{h}$ in $\mathfrak{g}$, i.e., a {\it reductive decomposition}
$
\mathfrak{g}=\mathfrak{h}+\mathfrak{m},
$ for $G/H$. In the Lie algebraic level, that means $[\mathfrak{h},\mathfrak{m}]\subset\mathfrak{m}$. Generally speaking, reductive decompositions for
$G/H$ are not unique.

A chosen reductive decomposition $\mathfrak{g}=\mathfrak{h}+\mathfrak{m}$ helps us study the homogeneous geometry of $G/H$, because
$\mathfrak{m}$ can be canonically identified as the tangent space $T_o(G/H)$, which is equivariant in the sense that the $\mathrm{Ad}(H)$-action on $\m$ coincides with the isotropy action on $T_o(G/H)$. Any vector $v\in\mathfrak{g}$ can be decomposed as $v=v_\mathfrak{h}+v_\mathfrak{m}$  according to the given reductive decomposition. Its $\mathfrak{m}$-summand coincides with the evaluation $V(o)$ for the vector field $V$ induced by  $v\in\mathfrak{g}$, i.e.,
$V(x)=\frac{d}{dt}|_{t=0}\exp tv\cdot x$.

 A homogeneous Finsler metric $F$ on $G/H$ can be one-to-one determined by its restriction $F(o,\cdot)$ to $T_o(G/H)=\m$, which is an arbitrary $\mathrm{Ad}(H)$-invariant Minkowski norm (for simplicity, we still use the same $F$ to denote this norm). The homogeneous metric $F$ is Riemannian if and only if it provides an $\mathrm{Ad}(H)$-invariant Euclidean norm $|\cdot|=\langle\cdot,\cdot\rangle^{1/2}$ on $\mathfrak{m}$.
 For a homogeneous metric $F$ on $G/H$, the vector field induced by any $v\in\mathfrak{g}$ is a Killing vector field.

When studying the local geometry of a homogeneous Finsler manifold, we may assume the connectedness and simple connectedness for $M=G/H$, and the connectedness for both $G$ and $H$. In this context, a homogeneous $(\alpha_1,\alpha_2)$ metric on $G/H$ can be determined by any arbitrary $\mathrm{Ad}(H)$-invariant $(\alpha_1,\alpha_2)$ norm on $\mathfrak{m}$. An $\mathrm{Ad}(H)$-invariant $(\alpha_1,\alpha_2)$ norm on $\mathfrak{m}$ can be determined by the following procedure.
Firstly we choose an $\mathrm{Ad}(H)$-invariant Euclidean norm $\alpha$ on $\mathfrak{m}$, and an $\mathrm{Ad}(H)$-invariant and $\alpha$-orthogonal decomopsition $\mathfrak{m}=\mathbf{V}_1+\mathbf{V}_2$ with each $n_i=\dim\mathbf{V}_i>0$. Nextly we choose some positively 1-homogeneous smooth real function $L:[0,+\infty)\times[0,+\infty)\rightarrow[0,+\infty)$ (the requirements for the function $L$ are implied by Theorem 3.2 in \cite{DX2016} or Theorem A in \cite{Xu2021}). Then the wanted norm can be presented as
\begin{equation*}\label{006}
F(y)=\sqrt{L(\alpha_1(y)^2,\alpha_2(y)^2)}=\sqrt{L(\alpha(y_1)^2,\alpha(y_2)^2)},\quad\forall  y=y_1+y_2\mbox{ with }  y_i\in\mathbf{V}_i.
 \end{equation*}
%
%
\subsection{Natural reductiveness} \label{prepareforcls}

In homogeneous Riemannian geometry, the natural reductiveness is defined as following \cite{KN1963}.

\begin{Definition}
 The homogeneous Riemannian manifold $(M,g)$ or the metric $g$ is called {\it naturally reductive} for the homogeneous space representation $M=G/H$ with $G\subset I(M,g)$ and the reductive decomposition $\g=\h+\m$, if
\begin{eqnarray}\label{011}
  \langle x, [z,y]_m\rangle +\langle  [z,x]_m,y\rangle=0,\quad\forall x,y,z\in \m.
\end{eqnarray}
or equivalently
\begin{equation}\label{012}
\langle y,[y,x]_\m\rangle=0,\quad\forall x,y\in\m.
\end{equation}
Here $\langle\cdot,\cdot\rangle$ is the $\mathrm{Ad}(H)$-invariant inner product on $\m$ induced by $g$.
\end{Definition}

In this definition, the natural reductiveness for a homogeneous Riemannian manifold is more like an algebraic property, because it depends on
the representation $M=G/H$ and decomposition $\g=\h+\m$.  It can also be more geometrized as following.

\begin{Definition}\label{definition3}
The homogeneous Riemannian manifold $(M,g)$ or the metric $g$ is called {\it naturally reductive} if there exists a homogeneous space representation $M=G/H$ with $G\subset I(M,g)$ and a reductive decomposition $\g=\h+\m$ such that (\ref{011}) or (\ref{012}) is satisfied.
\end{Definition}


In homogeneous Finsler geometry, there are several equivalent definitions or descriptions for the natural reductiveness.
The first one was proposed by Latifi \cite{La2007}.

\begin{Definition}\label{definition2}
The homogeneous Finsler manifold $(M,F)$ or the metric $F$ is called {\it naturally reductive} for the homogeneous space representation $M=G/H$ with $G\subset I(M,F)$ and
the reductive decomposition $\mathfrak{g}=\mathfrak{h}+\mathfrak{m}$, if we have
\begin{equation}\label{015}
g_y([w, u]_\mathfrak{m}, v)+g_y([w, v]_\mathfrak{m}, u)+2C_y([w, y]_\mathfrak{m}, u, v)=0, \forall y\in\mathfrak{m}\backslash\{0\},u,v,w\in \mathfrak{m}.
\end{equation}
\end{Definition}

The second one was proposed by Deng and Hou \cite{DH2010}.

\begin{Definition}
\label{definition1}
 The homogeneous Finsler manifold $(M,F)$ or the metric $F$ is called naturally reductive for the homogeneous space representation $M=G/H$ with $G\subset I(M,F)$ and the reductive decomposition $\g=\h+\m$,
if there exists a homogeneous Riemannian metric $\mathrm{g}$ on $M=G/H$ such that $g$ has the same connection as $F$ and its natural reductiveness is achieved for the given  decomposition.
\end{Definition}

The equivalence between Definition \ref{definition2} and Definition \ref{definition1} was proved in \cite{ZYD2022}.
Naturally reductive Finsler manifold (or metric) without mentioning the homogeneous space representation or the reductive decomposition can be similarly defined as Definition \ref{definition3}.
From Definition \ref{definition1}, we can immediately see that a naturally reductive Finsler metric is Berwald.

L. Huang provided an equivariant description using his spray vector field \cite{Hu2015}. For a homogeneous Finsler manifold $(G/H,F)$ with a reductive decomposition $\g=\h+\m$, the spray vector field $\eta:\m\backslash\{0\}\rightarrow\m$ is the $\mathrm{Ad}(H)$-invariant smooth map determined by
$$g_y(\eta(y),u)=g_y( y,[u,y]_\mathfrak{m}),\quad\forall y.$$
Then a homogeneous Finsler manifold $(G/H,F)$ is naturally reductive for a given reductive decomposition if and only if the corresponding spray vector field is constantly 0 \cite{DH2010}\cite{Hu2015}.

\begin{remark}\label{remark-2}
To be more self contained, here we briefly recall the proof in \cite{ZYD2022} for the equivalence between Definition \ref{definition2} and Definition \ref{definition1}. In \cite{DH2010}, Theorem 3.2 proves that Latifi's natural reductiveness implies Deng-Hou's, and Theorem 3.1 indicates that Deng-Hou's natural reductiveness implies $\eta\equiv0$. To prove the Latifi's natural reductiveness from $\eta\equiv0$, we consider the smooth vector field $W(y)=\mathrm{ad}_\m(w)y=[w,y]_\m$ for any $w\in\m$ and denote $\rho_t=\exp (t\cdot\mathrm{ad}_\m(w))$ the one-parameter subgroup $W$ generates in $GL(\m)$. When $\eta\equiv0$, $W$ is tangent to the indicatrix $F=1$. So each $\rho_t$ is a linear isomorphism preserving $F$, which is called a linear isometry for $F$ in \cite{XM2021}. A linear isometry for $F$ is automatically an isometry for the Hessian metric $g_y$ on $\m\backslash0$. So the generating vector field $W$ is Killing vector field for $(\m\backslash0,g_y)$, which is the meaning of (\ref{015}) in Definition \ref{definition1}.
\end{remark}

\section{Proof of Theorem \ref{main-thm-1}}


Let $(M,F)$ be a non-Riemannian homogeneous $(\alpha_1,\alpha_2)$ manifold, which is naturally reductive for the homogeneous space representation $M=G/H$ and the reductive decomposition $\g=\h+\m$. Then we have an $\mathrm{Ad}(H)$-invariant inner product $\alpha=\langle\cdot,\cdot\rangle^{1/2}$ on $\mathfrak{m}$ and an $\alpha$-orthogonal $\mathrm{Ad}(H)$-invariant decomposition $\mathfrak{m}=\mathfrak{m}_1+\mathfrak{m}_2$, such that the Minkowski norm $F$ on $\mathfrak{m}$ can be presented as
$F=\sqrt{L(\alpha_1^2,\alpha_2^2)}$
for some positively 1-homogeneous smooth function $L=L(u,v):[0,+\infty)\times[0,+\infty)\rightarrow[0,+\infty)$, where $\alpha_i(y)=\alpha(y_i)$ for $y=y_1+y_2$ with $y_i\in\mathfrak{m}_i$. This Minkowski norm is non-Euclidean, so we have

\begin{lemma}\label{lemma-1} $\frac{\partial}{\partial u}L$ and $\frac{\partial}{\partial v}L$ are linearly independent functions.
\end{lemma}
\begin{proof}Suppose $\frac{\partial}{\partial v}L=c\frac{\partial}{\partial u}L$ for some constant $c$. By the positive 1-homogeneity of $L$, $L=u\frac{\partial}{\partial u}L+v\frac{\partial}{\partial v}L$, so we have
$L=(u+cv)\frac{\partial}{\partial u}L$, which can be easily solved when we fix $v=1$, i.e., $L(u,1)=c'\cdot(u+c)$ for some constant $c'$. Using the positive 1-homogeneity of $L$ again, we see that $L$ is a linear function, and this is a contradiction.
\end{proof}

\begin{lemma}\label{lemma-2} Keeping all above assumptions and notations, then we have:
\begin{enumerate}
\item
each $\mathfrak{k}_i=\h+\m_i$ is a Lie subalgebra of $\mathfrak{g}$, i.e.,
$[\mathfrak{m}_i,\mathfrak{m}_i]\subset\mathfrak{h}+
    \mathfrak{m}_i$ for $i=1,2$;
\item
  $[\mathfrak{m}_1,\mathfrak{m}_2]\subset
\mathfrak{h}$;
\item for any $y_i\in\m_i$, $\langle y_i,[y_i,\m_i]_{\m_i}\rangle=0$, where the subscript ${\m_i}$ denotes the projection with respect to decomposition $\mathfrak{k}_i=\mathfrak{h}+\mathfrak{m}_i$.
\end{enumerate}
\end{lemma}
\begin{proof}
Let $e_1,\cdots,e_n$ be any $\alpha$-orthonormal basis of $\mathfrak{m}$ such that $e_i\subset\mathfrak{m}_1$ for $i\leq n_1$ and $e_i\subset\mathfrak{m}_2$ for $i>n_1$.
Denote $c^k_{ij}$'s the coefficients in $[e_i,e_j]_\mathfrak{m}=c_{ij}^k e_k$.
Let $y=ae_1+be_n\in\mathfrak{m}$ be a nonzero vector. We use $L_1$, $L_{2}$, $L_{12}$, etc, to denote the partial derivatives of $L(u,v)$ evaluated at $(a^2,b^2)$. Then all the fundamental tensor coefficients of $F=\sqrt{L(\alpha_1^2,\alpha_2^2)}$ at $y$ are the following (see \cite{DX2016}):
\begin{eqnarray*}
&&g_{11}=L_1+2a^2L_{11},\quad
g_{nn}=L_2+2b^2L_{22},\quad
g_{1n}=2abL_{12},\\
&&g_{ii}=L_1,\ \forall 1\leq i\leq n_1,\quad
g_{ii}=L_2,\ \forall n_1<i\leq n,\quad g_{ij}=0\ \mbox{otherwise}.
\end{eqnarray*}
From this calculation we can easily see the positiveness of $L_1$, $L_2$,
$L_1+2a^2L_{11}$ and $L_2+2b^2L_{22}$ for all values of $a$ and $b$.

By the naturally reductive property of $F$, we have for each $i\in\{1,\cdots,n\}$ and each nonzero $y=ae_1+be_n$,
\begin{eqnarray}0&=&
g_y( y,[y,e_i]_\mathfrak{m})\nonumber\\&=&
g_y( ae_1+be_n, (ac_{1i}^1+bc_{ni}^1)e_1+
(ac_{1i}^n+bc_{ni}^n)e_n)\nonumber\\
&=& c_{1i}^1a^2g_{11}+ c_{1i}^n a^2 g_{1n}+
c_{ni}^1 ab g_{11}+ c_{ni}^n ab g_{1n}\nonumber\\
& &+c_{1i}^1 ab g_{1n}+ c_{1i}^n ab g_{nn}
+c_{ni}^1 b^2 g_{1n}+ c_{ni}^n b^2 g_{nn}\nonumber\\
&=& c^1_{1i} (a^2 g_{11}+ab g_{1n}) + c^n_{1i} (a^2 g_{1n}+ab g_{nn}) \nonumber\\
& &+c_{ni}^1(ab g_{11}+b^2 g_{1n})
+c_{ni}^n  (ab g_{1n}+b^2 g_{nn})\nonumber\\
&=&c^1_{1i}(a^2 g_{11}+ab g_{1n})
+c_{1i}^n ab(L_2+2(a^2 L_{12}+b^2 L_{22}))\nonumber\\
& &+c_{ni}^1 ab(L_1+2(a^2 L_{11}+ b^2 L_{12}))+c_{ni}^n(ab g_{1n}+b^2 g_{nn})\nonumber\\
&=&c^1_{1i}(a^2 g_{11}+ab g_{1n})
+ab (c_{1i}^n L_2+c_{ni}^1 L_1)+c_{ni}^n(ab g_{1n}+b^2 g_{nn}),\label{001}
\end{eqnarray}
in which the last equality uses the positive 0-homogeneity of $\frac{\partial}{\partial u}L$ and $\frac{\partial}{\partial v}L$, i.e., $$a^2L_{12}+b^2L_{22}=a^2L_{11}+b^2L_{12}=0.$$
When we have $y\in\mathfrak{m}_1\backslash\{0\}$, i.e, $a\neq0$ and $b=0$,  (\ref{001}) provides $c_{1i}^1=0$.
Similarly, when we choose $y$ from $\mathfrak{m}_2\backslash\{0\}$, we see $c_{ni}^n=0$. Then by the assumption that $F$ is non-Riemannian and Lemma \ref{lemma-1}, we get $c_{1i}^n=c_{ni}^1=0$.


For any $y_1\in\mathfrak{m}_1$ and $y_2\in\mathfrak{m}_2$, we can find an $\alpha$-orthogonal
basis $e_1,\cdots,e_n$ as indicated above, such that $y_1\in\mathbb{R}e_1$ and $y_2\in\mathbb{R}e_n$.
By the vanishing of $c_{1i}^n$ and $c_{ni}^1$ for each $i$, we have
\begin{equation}\label{003}
\langle y_1, [y_2,\mathfrak{m}]_\mathfrak{m}\rangle=
\langle y_2,[y_1,\mathfrak{m}]\rangle=0, \quad\forall y_i\in\m_i,
\end{equation}
which can be interpreted as
\begin{equation}\label{002}
[\mathfrak{m}_1,\mathfrak{m}]\subset\mathfrak{h}+\mathfrak{m}_1 \quad\mbox{and}\quad\mbox{and}\quad
[\mathfrak{m}_2,\mathfrak{m}]
\subset\mathfrak{h}+\mathfrak{m}_2.
\end{equation}
From (\ref{002}),  we see
$[\mathfrak{m}_i,\mathfrak{m}_i]\subset\mathfrak{h}+
\mathfrak{m}_i$ for each $i$, and
$[\mathfrak{m}_1,\mathfrak{m}_2]\subset
(\mathfrak{h}+\mathfrak{m}_1)\cap(\mathfrak{h}+\mathfrak{m}_2)
=\mathfrak{h}$. This proves (1) and (2) in Lemma \ref{lemma-2}.

On the other hand, the vanishing of $c^1_{1i}$ and $c^n_{ni}$ provides
$\langle y_i,[y_i,\m_i]_{\m}\rangle=0$ for any $y_i\in\m_i$.
By (1) in Lemma \ref{lemma-2}, $[\cdot,\cdot]_\m$ coincides with $[\cdot,\cdot]_{\m_i}$ when restricted to $\m_i\times\m_i$. This proves (3) in Lemma \ref{lemma-2}.
\end{proof}

Now we are ready to prove Theorem \ref{main-thm-1}.

\begin{proof}[Proof of Theorem \ref{main-thm-1} when $H$ is connected]
We identify each $\mathfrak{m}_i$ as a subspace in $T_o(G/H)$. Then we have the
$G$-invariant distribution $\mathcal{V}_i=\cup_{g\in G}g_*(\mathfrak{m}_i)\subset T(G/H)$ on $G/H$.
The decomposition $\mathfrak{m}=\mathfrak{m}_1+\mathfrak{m}_2$ is orthogonal with respect to the inner product $\langle\cdot,\cdot\rangle=\alpha^2(\cdot)$, so the linear bundle decomposition $T(G/H)=\mathcal{V}_1+\mathcal{V}_2$
is orthogonal at each point with respect to the homogeneous Riemannian metric $\alpha$ on $G/H$.

Firstly, we prove that each $\mathcal{V}_i$ is integrable.
The first claim in Lemma \ref{lemma-2} provides the Lie subalgebra $\mathfrak{k}_i=\mathfrak{h}+\mathfrak{m}_i$. Let $K_i$ be the Lie subgroup
of $G$ generated by $\mathfrak{k}_i$. Then $H$ is a closed subgroup in $K_i$ and the submanifold $K_i/H$ is the integral submanifold of $\mathcal{V}_i$ passing $o$. Using the homogeneity, we see that $\mathcal{V}_i$ generates the smooth foliation $\{gK_i/H,\forall g\in G\}$ on $G/H$.

Secondly, we consider the restriction of $\alpha$ to $K_i/H$, which is the homogeneous metric determined by the Euclidean norm $\alpha_i|_{\mathfrak{m}_i}$. For simplicity, we denote $(K_i/H,\alpha_i)$ this homogeneous Riemannian submanifold.
Obviously $\mathfrak{k}_i=\mathfrak{h}+\mathfrak{m}_i$ is
a reductive decomposition for $K_i/H$, and by (3) in Lemma \ref{lemma-2}, $(K_i/H,\alpha_i)$ is
naturally reductive with respect to this decomposition.

Nextly, we prove that each submanifold $gK_i/H$ is a totally geodesic submanifold in $(G/H,\alpha)$.
Any vector $v\in\mathfrak{g}$ induces a Killing vector field
$V$ on $(G/H,\alpha)$. In particular, when the vector $v$ is from $\mathfrak{m}_i$,  $V$ is tangent to $K_i/H$ and $V(o)=v\in \mathfrak{m}=T_o(G/H)$.
For any Killing vector fields $X,Y$, corresponding to $u,v\in \mathfrak{m}_1$ and any $w\in\mathfrak{m}_2$ respectively,
we have
\begin{eqnarray}
\langle(\nabla^\alpha_XY)(o),w\rangle=\frac12
(\langle[w,u]_\mathfrak{m},v\rangle+
\langle[w,v]_\mathfrak{m},u\rangle-
\langle[u,v]_\mathfrak{m},w\rangle),\label{004}
\end{eqnarray}
in which $\nabla^\alpha$ is the Levi-Civita connection of $\alpha$.
By (1) and (2) in Lemma \ref{lemma-2}, each summand in the right side of (\ref{004}) vanishes. It implies that the second fundamental tensor for $K_1/H$ in $(G/H,\alpha)$ vanishes at $o$. Using the $G$-actions, we see the second fundamental tensors of all $gK_1/H$ vanish everywhere. For $gK_2/H$, the argument is similar.

To summarize, we have two integrable totally geodesic foliations $\mathcal{V}_1$ and $\mathcal{V}_2$, which are $\alpha$-complements of each other, at each point of $G/H$. So $(M,\alpha)$ is locally a Riemannian product between integral submanifolds of $\mathcal{V}_1$ and $\mathcal{V}_2$.

At last, we only need to prove the integral submanifold of $\mathcal{V}_i$ is naturally reductive. By the $G$-invariancy, we only need to consider $K_i/H$, i.e., the integral submanifold of $\mathcal{V}_i$ passing $o$. The restriction of $\alpha$ to
$K_i/H$ is the homogeneous Riemannian metric $\alpha_i$ determined by the Euclidean norm
$\alpha_i$ on $\mathfrak{m}_i$. Obviously, $\mathfrak{k}_i=\h+\m_i$ is a reductive decomposition for $K_i/H$ and by (3) in Lemma \ref{lemma-2},
$(K_i/H,\alpha_i)$ is naturally reductive for this decomposition.
\end{proof}

\begin{proof}[Proof of Theorem \ref{main-thm-1} when $H$ is not connected] Denote $H_0$ be the identity component of $H$. Then $\g=\h+\m$ is also a reductive decomposition for $G/H_0$ and $(G/H,F)$ is locally isometric to $(G/H_0,\tilde{F})$, in which the homogeneous metric $\tilde{F}$ is determined by the same $(\alpha_1,\alpha_2)$ norm on $\m$ as $F$. So $(G/H_0,\tilde{F})$ is also homogeneous $(\alpha_1,\alpha_2)$ manifold and it is naturally reductive with respect to the same decomposition as $F$. We have just proved $(G/H_0,\tilde{F})$ is locally isometric to
an $f$-product between two naturally reductive Riemannian manifolds. So the same claim is valid for $(G/H,F)$ immediately.
\end{proof}

For a connected and simple connected manifold, Theorem \ref{main-thm-1} can be strengthened as following.

\begin{Corollary}\label{cor-1}
A connected and simply connected non-Riemannian homogeneous $(\alpha_1,\alpha_2)$ manifold $(M,F)$ is naturally reductive if and only if it is the $f$-product between two connected and simply connected naturally reductive Riemannian manifolds $(M_i,\alpha_i)$ such that the function $L(u,v)$ in $F=\sqrt{L(\alpha_1^2,\alpha_2^2)}$ is nonlinear.
\end{Corollary}


\begin{proof} Suppose that $(M,F)$ is a naturally reductive
$(\alpha_1,\alpha_2)$ manifold.
When $(M,F)$ is connected and simply connected, the local $\alpha$-orthogonal product in the proof of Theorem \ref{main-thm-1} is global. So we have a Riemannian product decomposition $(M,\alpha)=(M_1,\alpha_1)\times (M_2,\alpha_2)$, in which each $(M_i,\alpha_i)$ is a simply connected naturally reductive Riemannian manifold, and $F=\sqrt{L(\alpha_1^2,\alpha_2^2)}$ is an $f$-product metric. Since $F$ is non-Riemannian, the function $L(u,v)$ is nonlinear.

Suppose that we have the Riemannian product decomposition $(M,\alpha)=(M_1,\alpha_1)\times (M_2,\alpha_2)$, in which each $(M_i,\alpha_i)$ is naturally reductive with respect to the representation $M_i=G_i/H_i$ and the reductive decomposition $\mathfrak{g}_i=\mathfrak{h}_i+\mathfrak{m}_i$. Then we claim that $F=\sqrt{L(\alpha_1^2,\alpha_2^2)}$ is naturally reductive for the presentation $M=G/H$ and the reductive decomposition $\g=\h+\m$, where $G=G_1\times G_2$, $H=H_1\times H_2$ and $\m=\m_1+\m_2$.
For any nonzero $y=y_1+y_2$ with $y_i\in\mathfrak{m}_i$, we can find an $\alpha_1$-orthonormal basis $\{e_1,\cdots,e_{n_1}\}$ for $\mathfrak{m}_1$ such that $y_1\in\mathbb{R}e_1$ and
an $\alpha_2$-orthogonal basis $\{e_{n_1+1},\cdots,e_n\}$ for $\mathfrak{m}_2$ such that  $y_2\in\mathbb{R}e_n$. Denote $c_{ij}^k$ the bracket coefficients in $[e_i,e_j]_\mathfrak{m}=c_{ij}^k e_k$. By the
naturally reductiveness of each $(G_i/H_i,\alpha_i)$ and the Lie algebra direct sum decomposition $\mathfrak{g}=\mathfrak{g}_1\oplus\mathfrak{g}_2$, we can easily see
\begin{equation}\label{013}
c_{1i}^1=c_{ni}^n=c_{ni}^1=c_{1i}^n=0,\quad\forall i.
\end{equation}
The calculation (\ref{001}) provides
$\langle y,[y,e_i]_\mathfrak{m}\rangle_y=0$, $\forall y$, $\forall i$, So $F=\sqrt{L(\alpha_1^2,\alpha_2^2)}$ is naturally reductive.
\end{proof}

There are many examples of naturally reductive $(\alpha_1,\alpha_2)$ manifolds. For example, when each $G_i/H_i$ is a normal homogeneous Riemannian manifold, the $f$-product metrics on $G_1/H_1\times G_2/H_2$ is naturally reductive. Most of these metrics are not Riemannian, locally Minkowski, or symmetric.

We also have the following byproduct.

\begin{Corollary}\label{cor-2}
Suppose the homogeneous non-Riemannian $(\alpha_1,\alpha_2)$ metric $F=\sqrt{L(\alpha_1^2,\alpha_2^2)}$ is naturally reductive
on $G/H$ with respect to the reductive decomposition
$\g=\h+\m$. Then $\alpha$ is also naturally reductive for the same decomposition.
\end{Corollary}

\begin{proof} Let $y=y_1+y_2$
be any nonzero vector in $\m$ with $y_i\in\m_i$. We choose the $\alpha$-orthonormal basis $\{e_1,\cdots,e_{n_1}\}$ for $\m_1$ with $y_1\in\mathbb{R}e_1$ and the $\alpha$-orthonormal basis $\{e_{n_1+1},\cdots,e_n\}$ for $\m_2$ with $y_2\in\mathbb{R}e_n$. By the argument in the proof of Lemma \ref{lemma-2},  we see the bracket coefficients
$c^k_{ij}$ in $[e_i,e_j]_\m=c^k_{ij}e_k$ still satisfy (\ref{013}).
Then direct calculation shows $\langle y,[y,\m]_\m\rangle=0$
for the inner product $\langle\cdot,\cdot\rangle=\alpha^2(\cdot)$ on $\mathfrak{m}$. So $\alpha$ is naturally reductive.
\end{proof}

\begin{remark} The natural reductiveness of a non-Riemannian $(\alpha_1,\alpha_2)$ metric $F$ on $G/H$ does not imply local product decompositions $G=G_1\times G_2$ and $H=H_1\times H_2$ such that $G_1/H_1$ and $G_2/H_2$ are the two factors in the local $f$-product representation for $F$.
Here is an example.
\end{remark}

\begin{exam}\label{example}
Let $(M_i,\alpha_i)$  be a standard Euclidean plane for $i=1,2$. Its connected isometry group $G_i$ is the semi product between $H_i=SO(2)$ for rotations and $\mathbb{R}^2$ for parallel translations. Correspondingly, we have a canonical reductive decomposition $\mathfrak{g}_i=\mathfrak{h}_i+\mathfrak{m}_i$. Let $H$ be a diagonal $SO(2)$ in $H_1\times H_2$ and $G$ the Lie subgroup of $G_1\times G_2$ with Lie algebra $\g=\h+\m_1+\m_2$. Then an $f$-product metric $F=\sqrt{L(\alpha_1^2,\alpha_2^2)}$ on $M_1\times M_2=G/H$ is naturally reductive with respect to the reductive decomposition $\g=\h+\m$ with $\m=\m_1+\m_2$. But $G$ is not locally a product of two Lie subgroups.
\end{exam}

\section{Curvatures of homogeneous $(\alpha_1,\alpha_2)$ manifolds}

\subsection{Proof of Theorem \ref{main-thm-2}} \label{prepareforcls}

The S-curvature curvature formula of a homogeneous manifold is the following (see Theorem 4.2 in
\cite{XD2015} and Proposition 4.6 in \cite{Hu2015}).

\begin{theorem}\label{preliminary-theorem}
Let $G/H$ be a smooth coset space with a
reductive decomposition $\g=\h+\m$. Then for any homogeneous Finsler $F$ and an invariant smooth
measure, the S-curvature $S(o,\cdot):\m\backslash0=T_o(G/H)
\backslash0\rightarrow\mathbb{R}$ is determined by
\begin{equation*}
S(o,y)=\langle y,\nabla^{g_{ij}}\ln\sqrt{\det(g_{pq})}(y)\rangle_y=I_y(\eta(y)),
\end{equation*}
where $\eta$ is the spray vector field and $I_y$ is the mean Cartan tensor.
\end{theorem}

Using Theorem \ref{preliminary-theorem}, we can prove

\begin{lemma}\label{lemma-3}
Let $G/H$ be a smooth coset space with a reductive decomposition $\g=\h+\m$, and the decomposition $\m=\m_1+\m_2$ is used for defining the homogeneous $(\alpha_1,\alpha_2)$ metric $F$ on $G/H$. Then for any nonzero vector $y\in\m_1\cup\m_2\subset\m=T_o(G/H)$, the S-curvature $S(o,y)$ vanishes.
\end{lemma}

\begin{proof}
Since $F$ is a homogeneous $(\alpha_1,\alpha_2)$ metric, the corresponding
Minkowski norm $F=F(o,\cdot)$ on $\m$ is invariant for
a standard block-diagonal action of $O(n_1)\times O(n_2)$, such that $\m_1$ and $\m_2$ are the fixed point set of the $O(n_2)$ and $O(n_1)$-actions respectively. Because of this symmetry, the Cartan tensor vanishes on $(\m_1\cup\m_2)\backslash\{0\}$ (the detailed argument is contained in the proof of Lemma 2.2 in \cite{XM2021}). So for any nonzero $y\in\m_1\cup\m_2$, $I_y(\cdot)\equiv0$, and then by Theorem \ref{preliminaries}, $S(o,y)=I_y(\eta(y))=0$.
\end{proof}

Now we prove Theorem \ref{main-thm-2}.

\begin{proof}[Proof of Theorem \ref{main-thm-2}]
The equality (3.19) in \cite{CS2012} points out the following basic fact,
$$E_{ij}=\frac12 S_{y^iy^j}=\frac12\frac{\partial^3 \mathbf{G}^k}{\partial y^i\partial y^j\partial y^k},$$
from which we see
\begin{eqnarray}\label{009}
S=(n+1)c(x)F+\epsilon \Longleftrightarrow E=\frac{n+1}{2}c(x)F^{-1}h,
\end{eqnarray}
i.e., a Finsler metric has weak isotropic S-curvature if and only if it has isotropic E-curvature.
This observation proves (2)$\Rightarrow$(4) and (3)$ \Rightarrow$(1) immediately. The statement (4)$\Rightarrow$(3) is trivial. So we only need to prove
(1)$\Rightarrow$(2).

Let $(G/H,F)$ be a homogeneous $(\alpha_1,\alpha_2)$ manifold with weakly isotropic S-curvature. Suppose $\g=\h+\m$ a reductive decomposition for $G/H$ and the decomposition $\m=\m_1+\m_2$ is used when defining $F$.
Since $F$ is reversible, so its S-curvature $S(x,y)$ is odd for the $y$-entry, i.e.,
$S(x,y)+S(x,-y)\equiv0$. Then the property of weakly isotropic S-curvature, $S(x,y)=(n+1)c(x)F(x,y)+\epsilon(x,y)$, for some function $c(x)$ and one-form $\epsilon(x,y)$, provides
$$S(x,y)+S(x,-y)=(n+1)c(x)(F(x,y)+F(x,-y)=2(n+1)c(x)F(x,y)=0,$$
i.e., $c(x)\equiv0$ and $S(x,y)=\epsilon(x,y)$. By Lemma \ref{lemma-3}, $S(o,\cdot)$ vanishes on
$(\m_1\cup\m_2)\backslash\{0\}$. Then by the linearity of $S(o,\cdot)=\epsilon(o,\cdot)$, $S(o,\cdot)$ vanishes at any nonzero $y\in\m=\m_1+\m_2$. Using the homogeneity, we see $(G/H,F)$ has vanishing S-curvature.
\end{proof}

\begin{remark}
Theorem 4.2 in \cite{DX2016} provides an
explicit S-curvature formula for a homogeneous $(\alpha_1,\alpha_2)$ metric. We may use it to observe
directly that the S-curvature $S(o,\cdot)$ is odd and vanishes
on $(\m_1\cup\m_2)\backslash0$. Theorem 4.3 in \cite{DX2016} provides the algebraic characterization of the vanishing S-curvature property for a homogeneous $(\alpha_1,\alpha_2)$ metric, i.e., a homogeneous $(\alpha_1,\alpha_2)$ metric $F$ on $G/H$ with a reductive decomposition $\g=\h+\m$
has vanishing S-curvature if and only if
\begin{equation}\label{014}
\langle [y_1,\m_2]_\m,y_1\rangle=
\langle y_2,[y_2,\m_1]_\m\rangle=0,\quad\forall y_1\in\m_1,y_2\in\m_2,
\end{equation}
where the inner product $\langle\cdot,\cdot\rangle=\alpha^2(\cdot)$ and
the decomposition $\m=\m_1+\m_2$ are used for defining
$F$. By Theorem \ref{main-thm-2}, (\ref{014}) also characterize the properties of vanishing E-curvature, isotropic E-curvature, etc.
\end{remark}

\subsection{ Flag curvature of a naturally reductive $(\alpha_1,\alpha_2)$ manifold}
In this section, we calculate the flag curvature of a naturally reductive $(\alpha_1, \alpha_2)$ metric. Let $F$ be a naturally reductive $(\alpha_1, \alpha_2)$ metric on $G/H$ with
respect to the reductive decomposition $\mathfrak{g}=\mathfrak{h}+\mathfrak{m}$. We present $F$ as $F(y)=|y|\phi(\tfrac{|y_2|}{|y|})$. Here $|\cdot|=\langle\cdot,\cdot\rangle^{1/2}$ is an $\mathrm{Ad}(H)$-invariant Euclidean norm on $\m$. For any $y\in\m$, we have $y=y_1+y_2$ with $y_i\in\m_i$ for an $\mathrm{Ad}(H)$-invariant $\langle\cdot,\cdot\rangle$-orthogonal decomposition $\m=\m_1+\m_2$.

\begin{theorem}\label{main-thm-3}
Keep all assumptions and notations in this section, and let $(o,y,\mathbf{P})$ be a
the flag triple such that $\{y,x\}$ is a $\langle\cdot,\cdot\rangle$-orthonormal basis of $P\subset\m$.
Then the flag curvature of $(o,y,\mathbf{P})$ is
\begin{eqnarray*}
K(o,y,\mathbf{P}) &=& \frac{[\phi^{2}(|y_2|)-|y_2|\cdot\phi(|y_2|)\phi^{'}(|y_2|)]N+ \frac{1}{|y_2|}\phi(|y_2|)\phi^{'}(|y_2|)Q}{\phi^3(|y_2|)[\phi(|y_2|)+ \phi^{'}(|y_2|)M+\frac{\langle x_2, y_2\rangle^2}{\langle y_2, y_2\rangle}\phi^{''}(|y_2|)]},
\end{eqnarray*}
where
\begin{eqnarray*}
  M &=& \frac{\langle x_2, x_2\rangle}{|y_2|}-\frac{\langle x_2, y_2\rangle^2}{|y_2|^3}-|y_2|,\\
  N &=& \frac{1}{4}\langle [x,y]_{\m}, [x,y]_{\m}\rangle+\langle[[x,y]_{\h}, x], y\rangle, \\
  Q &=& \frac{1}{4}\langle [x_2,y]_{\m}, [x,y]_{\m}\rangle+\langle[[x,y]_{\h}, x_2], y\rangle,
\end{eqnarray*}
 and $x_2, y_2$ denote the components of $x, y$ in $\m_2$ respectively,
\end{theorem}
\begin{proof}
For any $y\in \m\backslash\{0\}$ and $u,v\in\m$, denote $y=y_1+y_2, u=u_1+u_2$ and $v=v_1+v_2$ with $y_i,u_i,v_i\in \m_i, i=1,2$. Let $F$ be a Minkowski norm on $\m$ and $g$ the fundamental tensor of $F$. We write
$F(y)= |y|^2\phi(\frac{|y_2|}{|y|})$. So
\[
F^2(y+su+tv)=\langle y+su+tv, y+su+tv \rangle\phi^2(\frac{\sqrt{\langle y_2+su_2+tv_2, y_2+su_2+tv_2\rangle}}{\sqrt{\langle y+su+tv, y+su+tv\rangle}}),
\]
then we have
\begin{eqnarray}\label{Formula}
  g_y(u,v) &=& \frac{1}{2}\frac{\partial^2}{\partial s\partial t}F^2(y+su+tv)|_{s=t=0} \nonumber \\
           &=& \langle u,v\rangle \phi^2(\frac{|y_2|}{|y|})+\phi(\frac{|y_2|}{|y|})\phi^{'}(\frac{|y_2|}{|y|})(
           \frac{\langle y_2,u_2\rangle \langle y,v\rangle }{|y||y_2|}+ \frac{\langle y_2,v_2\rangle \langle y,u\rangle }{|y||y_2|}-\frac{|y_2|\langle y, u \rangle \langle y,v\rangle}{|y|^3} \nonumber\\
           &&+\frac{\langle u_2,v_2\rangle|y|}{|y_2|}-\frac{\langle u,v\rangle|y_2|}{|y|}-\frac{|y|\langle y_2, u_2 \rangle \langle y_2,v_2\rangle}{|y_2|^3})+(\phi^{'2}(\frac{|y_2|}{|y|})+\phi(\frac{|y_2|}{|y|})\phi^{''}(\frac{|y_2|}{|y|}))\cdot \nonumber\\
           &&(\frac{\langle y_2,u_2\rangle}{|y_2|}-\frac{\langle y,u\rangle|y_2|}{|y|^2})(\frac{\langle y_2,v_2\rangle}{|y_2|}-\frac{\langle y,v\rangle|y_2|}{|y|^2}).
               \end{eqnarray}
From the above formula, we obtain
\begin{eqnarray}\label{6}
&&g_y(y,y)=\langle y,y\rangle \phi^2(\frac{|y_2|}{|y|})=F^2(y).
\end{eqnarray}
By the unit orthonormality of $x, y$, denote $x=x_1+x_2$,  and $y=y_1+y_2$, $x_i,y_i \in \m_i, i=1,2$, it is easy to see
\begin{eqnarray}\label{7}
&&g_y(y,x)=\frac{1}{|y_2|}\phi(|y_2|)\phi^{'}(|y_2|)\langle x_2, y_2\rangle.
\end{eqnarray}
By \ref{Formula}, we have
\begin{eqnarray}\label{8}
g_y(x,x)&=&\phi^2(|y_2|)+\phi(|y_2|)\phi^{'}(|y_2|)(\frac{\langle x_2, x_2\rangle}{|y_2|}- \frac{\langle x_2, y_2\rangle\langle x_2, y_2\rangle}{|y_2|^3}-|y_2|) \nonumber \\
&&+(\phi^{'2}(|y_2|)+\phi(|y_2|)\phi^{''}(|y_2|))(\frac{\langle x_2, y_2\rangle\langle x_2, y_2\rangle}{|y_2|^2}).
\end{eqnarray}
Moreover,
\begin{eqnarray}\label{9}
  g_y(R(x,y)y, x) &=& \phi^2(|y_2|)\langle R(x, y)y, x\rangle+\phi(|y_2|)\phi^{'}(|y_2|)(\frac{\langle [R(x, y)y]_2, x_2\rangle}{|y_2|}\nonumber\\
  &&-\frac{\langle y_2, x_2\rangle\langle [R(x, y)y]_2, y_2\rangle}{|y_2|^3}- |y_2|\langle R(x, y)y, x\rangle) \nonumber\\
  &&+(\phi^{'2}(|y_2|)+\phi(|y_2|)\phi^{''}(|y_2|))\cdot \frac{\langle [R(x, y)y]_2, y_2\rangle}{|y_2|}\cdot \frac{\langle y_2, x_2\rangle}{|y_2|},
     \end{eqnarray}
where $[R(x, y)y]_2$ denote the component of $R(x, y)y$ in $\m_2$.

According to Theorem 2.1 in \cite{DH2004}, the Riemann curvature $R_y(x)$ at $o\in(G/H,F)$ is given by
\begin{eqnarray*}
&&R_y(x)=R(x,y)y = -[[x,y]_{\h},y]-\frac{1}{4}[[x,y]_{\m},y]_{\m}.
\end{eqnarray*}
Thus, we obtain
\begin{eqnarray}\label{10}
 \langle R_y(x), x\rangle&=& \langle -\frac{1}{4}[[x,y]_{\m}, y]_{\m}, x\rangle- \langle[[x,y]_{\h}, y], x\rangle \nonumber \\
  &=& \frac{1}{4}\langle [x,y]_{\m}, [x,y]_{\m}\rangle+\langle[[x,y]_{\h}, x], y\rangle,
\end{eqnarray}
\begin{eqnarray}\label{11}
 \langle [R_y(x)]_2, x_2\rangle&=& \langle -\frac{1}{4}[[x,y]_{\m}, y]_{\m}, x_2\rangle- \langle[[x,y]_{\h}, y], x_2\rangle \nonumber \\
  &=& \frac{1}{4}\langle [x_2,y]_{\m}, [x,y]_{\m}\rangle+\langle[[x,y]_{\h}, x_2], y\rangle,
\end{eqnarray}
\begin{eqnarray}\label{12}
 \langle [R_y(x)]_2, y_2\rangle&=& \langle -\frac{1}{4}[[x,y]_{\m}, y]_{\m}, y_2\rangle- \langle[[x,y]_{\h}, y_2], y_2\rangle \nonumber \\
  &=& \frac{1}{4}\langle [y_2,y]_{\m}, [x,y]_{\m}\rangle+\langle[[x,y]_{\h}, y_2], y_2\rangle \nonumber \\
  &=&0.
\end{eqnarray}
substituting \ref{6}-\ref{12} into
\begin{eqnarray*}\label{13}
 K(o,y,\mathbf{P})&=& \frac{g_y(R_y(x), x)}{g_y(y,y)g_y(x,x)-g_y(x,y)^2},
\end{eqnarray*}
  we completes the proof.
\end{proof}

\noindent
{\bf Acknowledgement}.
This paper is supported by National Natural Science Foundation of China (12001007, 11821101, 12131012),
Beijing Natural Science Foundation (1222003, Z180004),
Natural Science Foundation of Anhui province (1908085QA03).

\end{document}